\documentclass[fleqn]{article}

\usepackage{url,amsmath,amsfonts,latexsym,amsthm,amssymb}

\title{Rosser provability and the second incompleteness theorem}
\author{Taishi Kurahashi}
\date{}

\theoremstyle{plain}
\newtheorem{theorem}{Theorem}[section]
\newtheorem{lemma}[theorem]{Lemma}
\newtheorem{proposition}[theorem]{Proposition}
\newtheorem{corollary}[theorem]{Corollary}

\theoremstyle{definition}
\newtheorem{definition}[theorem]{Definition}
\newtheorem{remark}[theorem]{Remark}

\newcommand{\PA}{{\sf PA}}
\newcommand{\PR}{{\rm Pr}}
\newcommand{\PRR}{{\rm Pr}^R}

\newcommand{\Prf}{{\rm Prf}}
\newcommand{\Con}{{\sf Con}}
\newcommand{\N}{\mathbb{N}}
\newcommand{\gdl}[1]{\ulcorner#1\urcorner}

\newcommand{\HB}[1]{{\bf HB#1}}
\newcommand{\D}[1]{{\bf D#1}}
\newcommand{\DC}{{\bf \Delta_0 C}}
\newcommand{\SC}{{\bf \Sigma_1 C}}
\newcommand{\GC}{{\bf \Gamma C}}
\newcommand{\BD}[1]{{\bf B_{#1}}}
\newcommand{\DU}[1]{{\bf D#1^U}}
\newcommand{\DCU}{{\bf \Delta_0 C^U}}
\newcommand{\SCU}{{\bf \Sigma_1 C^U}}
\newcommand{\GCU}{{\bf \Gamma C^U}}
\newcommand{\CB}{{\bf CB}}
\newcommand{\BDU}[1]{{\bf B_{#1}^U}}
\newcommand{\DG}[1]{{\bf D#1^G}}
\newcommand{\DCG}{{\bf \Delta_0 C^G}}
\newcommand{\SCG}{{\bf \Sigma_1 C^G}}
\newcommand{\GCG}{{\bf \Gamma C^G}}

\begin{document}

\maketitle

\begin{abstract}
This paper is a continuation of Arai's paper on derivability conditions for Rosser provability predicates. 
We investigate the limitations of the second incompleteness theorem by constructing three different Rosser provability predicates satisfying several derivability conditions. 
\end{abstract}

\section{Introduction}\label{sec-intro}

G\"odel's second incompleteness theorem states that if $T$ is a recursively axiomatized consistent extension of Peano Arithmetic $\PA$, then $T$ cannot prove the consistency of $T$. 
This statement of the theorem is somewhat ambiguous, and it should be stated more precisely. 
In 1939, the first detailed proof of the second incompleteness theorem appeared in their book \cite{HB39} by Hilbert and Bernays. 
They introduced the following conditions $\HB{1}$, $\HB{2}$ and $\HB{3}$ which are called the Hilbert-Bernays derivability conditions, and essentially proved that if $T$ is as above and a $\Sigma_1$ provability predicate $\PR_T(x)$ of $T$ satisfies the Hilbert-Bernays derivability conditions, then the consistency statement $\forall x (\PR_T(x) \to \neg \PR_T(\dot{\neg} x))$ of $T$ cannot be proved in $T$. 

\begin{description}
	\item [$\HB{1}$] : If $T \vdash \varphi \to \psi$, then $T \vdash \PR_T(\gdl{\varphi}) \to \PR_T(\gdl{\psi})$ for any formulas $\varphi, \psi$. 
	\item [$\HB{2}$] : $T \vdash \PR_T(\gdl{\neg \varphi(x)}) \to \PR_T(\gdl{\neg \varphi(\dot{x})})$ for any formula $\neg \varphi(x)$. 
	\item [$\HB{3}$] : $T \vdash f(x) = 0 \to \PR_T(\gdl{f(\dot{x}) = 0})$ for any primitive recursive term $f(x)$.
\end{description}

Moreover, Hilbert and Bernays proved that G\"odel's provability predicate $\PR_T(x)$ satisfies these conditions. 
In 1955, L\"ob \cite{Lob55} introduced the following conditions $\D{1}$, $\D{2}$ and $\D{3}$  which are called the Hilbert-Bernays-L\"ob derivability conditions, and proved that if $\PR_T(x)$ satisfies these conditions, then L\"ob's theorem holds. 

\begin{description}
	\item [$\D{1}$] : If $T \vdash \varphi$, then $T \vdash \PR_T(\gdl{\varphi})$ for any formula $\varphi$. 
	\item [$\D{2}$] : $T \vdash \PR_T(\gdl{\varphi \to \psi}) \to (\PR_T(\gdl{\varphi}) \to \PR_T(\gdl{\psi}))$ for any formulas $\varphi$, $\psi$. 
	\item [$\D{3}$] : $T \vdash \PR_T(\gdl{\varphi}) \to \PR_T(\gdl{\PR_T(\gdl{\varphi})})$ for any formula $\varphi$. 
\end{description}

L\"ob's theorem is known as a generalization of the second incompleteness theorem. 
Hence we obtain that if $\PR_T(x)$ satisfies the Hilbert-Bernays-L\"ob derivability conditions, then $T$ cannot prove the consistency statement $\neg \PR_T(\gdl{0 \neq 0})$ of $T$. 
This seems to be the most well-known form of the second incompleteness theorem stated accurately.

%
Other sets of derivability conditions which are sufficient for the second incompleteness theorem have been proposed by Jeroslow \cite{Jer73}, Montagna \cite{Mon79} and Buchholz \cite{Buc93} (see also \cite{Kur2}). 
On the other hand, the second incompleteness theorem does not hold for some provability predicates. 
Feferman \cite{Fef60} found a $\Pi_1$ formula $\tau(v)$ representing the set of all axioms of $T$ in $T$ such that the consistency statement $\forall x(\PR_\tau(x) \to \neg \PR_\tau(\dot{\neg} x))$ is provable in $\PA$ where $\PR_\tau(x)$ is the provability predicate of $T$ constructed from $\tau(v)$. 
Notice that Feferman's predicate satisfies $\D{1}$ and $\D{2}$, but does not satisfy $\D{3}$ because it is not $\Sigma_1$. 

An example of a $\Sigma_1$ provability predicate for which the second incompleteness theorem does not hold was given by Mostowski \cite{Mos65}. 
Let $\PR_T^M(x)$ be the $\Sigma_1$ formula $\exists y({\rm Prf}_T(x, y) \land \neg {\rm Prf}_T(\gdl{0 \neq 0}, y))$ where ${\rm Prf}_T(x, y)$ is a $\Delta_1(\PA)$ formula saying that ``$y$ is a $T$-proof of $x$''. 
Then $\neg \PR_T^M(\gdl{0 \neq 0})$ is trivially provable in $\PA$. 
Since the formula $\PR_T^M(x)$ satisfies $\D{1}$ and $\D{3}$, it does not satisfy $\D{2}$. 
Mostowski's example shows that for $\Sigma_1$ provability predicates $\PR_T(x)$, the set $\{\D{1}, \D{3}\}$ of derivability conditions is not sufficient for the unprovability of $\neg \PR_T(\gdl{0 \neq 0})$. 

Rosser provability predicates were introduced by Rosser \cite{Ros36} to improve G\"odel's first incompleteness theorem, and they are also examples of $\Sigma_1$ provability predicates for which the second incompleteness theorem does not hold (\cite{Kre60,Kre65,Kre71}). 
That is, $\PA \vdash \neg \PRR_T(\gdl{0 \neq 0})$ for any Rosser provability predicate $\PRR_T(x)$ of $T$. 
It follows that each Rosser provability predicate does not satisfy at least one of $\D{2}$ and $\D{3}$. 
It is known that whether each Rosser provability predicate satisfies $\D{2}$ (and $\D{3}$) or not depends on the choice of a Rosser predicate.  
Indeed, by using Kripke model theoretic method by Guaspari and Solovay \cite{GS79}, we obtain a Rosser provability predicate satisfying neither $\D{2}$ nor $\D{3}$. 
Also Bernardi and Montagna \cite{BM84} and Arai \cite{Ara90} proved the existence of Rosser predicates satisfying $\D{2}$, and Arai proved the the existence of Rosser predicates satisfying $\D{3}$. 

Moreover, it can be shown that the consistency statement $\forall x(\PRR_T(x) \to \neg \PRR_T(\dot{\neg}x))$ is provable for each Arai's Rosser provability predicate $\PRR_T(x)$. 
Then Arai's results indicate that for $\Sigma_1$ provability predicates $\PR_T(x)$, each of $\{\D{1}, \D{2}\}$ and $\{\D{1}, \D{3}\}$ is not sufficient for the unprovability of $\forall x(\PRR_T(x) \to \neg \PRR_T(\dot{\neg}x))$. 
Also these existence results show that $\{\D{1}, \D{2}\}$ and $\{\D{1}, \D{3}\}$ do not imply $\D{3}$ and $\D{2}$, respectively. 

The constructions of Rosser provability predicates are somewhat flexible, and thus actually, Rosser provability predicates satisfying several derivability conditions have also been investigated (\cite{KK17,Kur,Kur16}).
In this paper, we construct three Rosser provability predicates satisfying several additional derivability conditions. 
As a consequence of these constructions, we obtain that some sets of conditions of provability predicates are not sufficient for some versions of the second incompleteness theorem. 
In particular, our second and third Rosser provability predicates satisfy the Hilbert-Bernays derivability conditions. 
Therefore in contrast to the Hilbert-Bernays-L\"ob derivability conditions, the Hilbert-Bernays derivability condition does not imply the unprovability of the consistency statement $\neg \PR_T(\gdl{0 \neq 0})$ in general. 

In Section \ref{sec-dc}, we introduce versions of derivability conditions, and also introduce some basic results from the paper \cite{Kur2} concerning derivability conditions. 
In Section \ref{sec-rpp}, we introduce Rosser provability predicates, and describe background of the present paper. 
In the last section, we give constructions of our Rosser provability predicates.

\section{Provability predicates and derivability conditions}\label{sec-dc}

Throughout this paper, $T$ denotes a recursively axiomatized consistent extension of Peano arithmetic $\PA$ in the language of first-order arithmetic $\mathcal{L}_A$. 
The numeral for each natural number $n$ is denoted by $\overline{n}$. 
We fix some natural G\"odel numbering, and let $\gdl{\varphi}$ be the numeral for the G\"odel number of a formula $\varphi$. 
We assume that $0$ is not a G\"odel number of any object. 
Let $\{\xi_k\}_{k \in \omega}$ be the effective reputation-free sequence of all $\mathcal{L}_A$-formulas arranged in ascending order of G\"odel numbers. 
We assume that if $\xi_k$ is a proper subformula of $\xi_l$, then $k < l$. 

We say a $\Sigma_1$ formula $\PR_T(x)$ is a {\it provability predicate} of $T$ if it weakly represents the set of all theorems of $T$ in $T$, that is, for any natural number $n$, $T \vdash \PR_T(\overline{n})$ if and only if $n$ is the G\"odel number of some theorem of $T$. 
Provability predicates are expected to satisfy some natural conditions which are called derivability conditions. 
We introduce three versions of derivability conditions, that is, local version, uniform version and global version. 
See \cite{Kur2} for further details. 
In the following definitions, let $\Gamma$ be either $\Delta_0$ or $\Sigma_1$. 

\begin{definition}[Local derivability conditions]
\leavevmode
\begin{description}
	\item [$\D{1}$ :] If $T \vdash \varphi$, then $T \vdash \PR_T(\gdl{\varphi})$ for any formula $\varphi$. 
	\item [$\D{2}$ :] $T \vdash \PR_T(\gdl{\varphi \to \psi}) \to (\PR_T(\gdl{\varphi}) \to \PR_T(\gdl{\psi}))$ for any formulas $\varphi$, $\psi$. 
	\item [$\D{3}$ :] $T \vdash \PR_T(\gdl{\varphi}) \to \PR_T(\gdl{\PR_T(\gdl{\varphi})})$ for any formula $\varphi$. 
	\item [$\GC$ :] $T \vdash \varphi \to \PR_T(\gdl{\varphi})$ for any $\Gamma$ sentence $\varphi$. 
	\item [$\BD{2}$ :] If $T \vdash \varphi \to \psi$, then $T \vdash \PR_T(\gdl{\varphi}) \to \PR_T(\gdl{\psi})$ for any formulas $\varphi$, $\psi$. 
\end{description}
\end{definition}

\begin{definition}[Uniform derivability conditions]
\leavevmode
\begin{description}
	\item [$\DU{1}$ :] If $T \vdash \forall \vec{x} \varphi(\vec{x})$, then $T \vdash \forall \vec{x} \PR_T(\gdl{\varphi(\vec{\dot{x}})})$ for any formula $\varphi(\vec{x})$. 
	\item [$\DU{2}$ :] $T \vdash \forall \vec{x}(\PR_T(\gdl{\varphi(\vec{\dot{x}}) \to \psi(\vec{\dot{x}})}) \to (\PR_T(\gdl{\varphi(\vec{\dot{x}})}) \to \PR_T(\gdl{\psi(\vec{\dot{x}})})))$ for any formulas $\varphi(\vec{x})$, $\psi(\vec{x})$. 
	\item [$\DU{3}$ :] $T \vdash \forall \vec{x}(\PR_T(\gdl{\varphi(\vec{\dot{x}})}) \to \PR_T(\gdl{\PR_T(\gdl{\varphi(\vec{\dot{x}})})}))$ for any formula $\varphi(\vec{x})$. 
	\item [$\GCU$ :] $T \vdash \forall \vec{x}(\varphi(\vec{x}) \to \PR_T(\gdl{\varphi(\vec{\dot{x}})}))$ for any $\Gamma$ formula $\varphi(\vec{x})$. 
	\item [$\BDU{2}$ :] If $T \vdash \forall \vec{x}(\varphi(\vec{x}) \to \psi(\vec{x}))$, then $T \vdash \forall \vec{x} (\PR_T(\gdl{\varphi(\vec{\dot{x}})}) \to \PR_T(\gdl{\psi(\vec{\dot{x}})}))$ for any formulas $\varphi(\vec{x})$, $\psi(\vec{x})$. 
	\item [$\CB$ :] $T \vdash \PR_T(\gdl{\forall \vec{x}\varphi(\vec{x})}) \to \forall \vec{x}\PR_T(\gdl{\varphi(\vec{\dot{x}})})$ for any formula $\varphi(\vec{x})$. 
\end{description}
\end{definition}

\begin{definition}[Global derivability conditions]
\leavevmode
\begin{description}
	\item [$\DG{2}$ :] $T \vdash \forall x \forall y(\PR_T(x \dot{\to} y) \to (\PR_T(x) \to \PR_T(y)))$. 
	\item [$\DG{3}$ :] $T \vdash \forall x(\PR_T(x) \to \PR_T(\gdl{\PR_T(\dot{x})}))$. 
	\item [$\GCG$ :] $T \vdash \forall x({\sf True}_\Gamma(x) \to \PR_T(x))$. 
\end{description}
\end{definition}

Here $\gdl{\varphi(\vec{\dot{x}})}$ is an abbreviation for $\gdl{\varphi(\dot{x}_0, \ldots, \dot{x}_{k-1})}$ which is a primitive recursive term corresponding to a primitive recursive function calculating the G\"odel number of $\varphi(\overline{n_0}, \ldots, \overline{n_{k-1}})$ from $n_0, \ldots, n_{k-1}$.
Also $x \dot{\to} y$ is a primitive recursive term such that $\PA \vdash \gdl{\varphi} \dot{\to} \gdl{\psi} = \gdl{\varphi \to \psi}$ for all formulas $\varphi$ and $\psi$. 
Furthermore ${\sf True}_\Gamma(x)$ is a natural formula defining the truth of $\Gamma$ sentences (cf.~H\'ajek and Pudl\'ak \cite{HP93}). 

Notice that every provability predicate automatically satisfies \D{1}. 
Since our provability predicates are $\Sigma_1$, $\D{3}$ is a particular case of $\SC$. 
Also $\DU{3}$ and $\DG{3}$ are particular cases of $\SCU$. 
The condition $\CB$ claims the provability of sentences corresponding to the Converse Barcan Formula (see \cite{HC}). 
It is easy to prove the following implications (see \cite{Kur2}). 

\begin{proposition}\label{DCP}
\leavevmode
\begin{enumerate}
	\item $\DC$ and $\BD{2} \Rightarrow \D{1}$. 
	\item $\DCU$ and $\BD{2} \Rightarrow \DU{1}$. 
	\item $\D{1}$ and $\D{2} \Rightarrow \BD{2}$. 
	\item $\DU{1}$ and $\DU{2} \Rightarrow \BDU{2}$. 
	\item $\D{1}$ and $\CB \Rightarrow \DU{1}$. 
	\item $\BDU{2} \Rightarrow \CB$. 
\end{enumerate}
\end{proposition}

Here the first clause of Proposition \ref{DCP} means that for any $\Sigma_1$ formula $\PR_T(x)$, if $\PR_T(x)$ satisfies both $\DC$ and $\BD{2}$, then it also satisfies $\D{1}$. 

Moreover, the following nontrivial implication holds. 

\begin{theorem}[Kurahashi \cite{Kur2}]\label{TK}
$\D{1}$ and $\BDU{2} \Rightarrow \SCU$. 
\end{theorem}

By Proposition \ref{DCP}.4, we immediately obtain the following corollary which is due to Buchholz (see also \cite{Rau10}). 

\begin{corollary}[Buchholz \cite{Buc93}]\label{BucT}
$\DU{1}$ and $\DU{2} \Rightarrow \SCU$. 
\end{corollary}

We introduce several different consistency statements based on the provability predicate $\PR_T(x)$. 

\begin{definition}\leavevmode
\begin{itemize}
	\item $\Con_{\PR_T}^H : \equiv \forall x (\PR_T(x) \to \neg \PR_T(\dot{\neg} x))$. 
	\item $\Con_{\PR_T}^L : \equiv \neg \PR_T(\gdl{0 \neq 0})$. 
	\item For each formula $\varphi$, $\Con_{\PR_T}(\varphi) : \equiv (\PR_T(\gdl{\varphi}) \to \neg \PR_T(\gdl{\neg \varphi}))$. 
	\item $\Con_{\PR_T}^S : = \{\Con_{\PR_T}(\varphi) : \varphi$ is a formula$\}$.\footnote{Introducing this schematic consistency statement $\Con_{\PR_T}^S$ was proposed by the referee.}
\end{itemize}
\end{definition}

Here $\dot{\neg} x$ is a primitive recursive term satisfying $\PA \vdash \dot{\neg} \gdl{\varphi} = \gdl{\neg \varphi}$ for any formula $\varphi$. 
Then for every formula $\varphi$, $\Con_{\PR_T}(\varphi)$ follows from $\Con_{\PR_T}^H$. 
Also $\Con_{\PR_T}^L$ follows from $\Con_{\PR_T}(0=0)$ because $\PR_T(x)$ satisfies $\D{1}$.  
Therefore we have that $\Con_{\PR_T}^H$ is stronger than $\Con_{\PR_T}^S$, and $\Con_{\PR_T}^S$ is stronger than $\Con_{\PR_T}^L$. 
In general, the converse implications do not hold. 

Hilbert and Bernays introduced the conditions $\HB{1}$, $\HB{2}$ and $\HB{3}$ described in the introduction which are sufficient for unprovability of $\Con_{\PR_T}^H$. 
In our context, each of their conditions correponds to $\BD{2}$, $\CB$ and $\DCU$, respectively. 
Then we call the conditions $\BD{2}$, $\CB$ and $\DCU$ the {\it Hilbert-Bernays derivability conditions}. 
Their result can be stated as follows (see \cite{Kur2}). 

\begin{theorem}[Hilbert and Bernays \cite{HB39}]\label{HB}
If $\PR_T(x)$ satisfies $\BD{2}$, $\CB$ and $\DCU$, then $T \nvdash \Con_{\PR_T}^H$. 
\end{theorem}

The following theorem is a well-known form of unprovability of consistency which is essentially due to L\"ob. 
The conditions $\D{1}$, $\D{2}$ and $\D{3}$ are called the {\it Hilbert-Bernays-L\"ob derivability conditions}. 

\begin{theorem}[L\"ob \cite{Lob55}]\label{Lob}
If $\PR_T(x)$ satisfies $\D{1}$, $\D{2}$ and $\D{3}$, then $T \nvdash \Con_{\PR_T}^L$. 
\end{theorem}

By Proposition \ref{DCP}.3, $\{\D{1}, \D{2}\}$ implies $\{\D{1}, \BD{2}\}$. 
In \cite{Kur2}, it is proved that if $\PR_T(x)$ satisfies $\D{1}$, $\BD{2}$ and $\D{3}$, then $T \nvdash \Con_{\PR_T}^H$. 
This statement can be strengthened as follows. 

\begin{theorem}\label{KurG2}
If $\PR_T(x)$ satisfies $\D{1}$, $\BD{2}$ and $\D{3}$, then $T \nvdash \Con_{\PR_T}^S$.\footnote{This means that $T \nvdash \Con_{\PR_T}(\varphi)$ for some formula $\varphi$.} 
\end{theorem}
\begin{proof}
Suppose $\PR_T(x)$ satisfies $\D1$, $\BD{2}$ and $\D{3}$. 
Let $\varphi$ be any sentence satisfying $T \vdash \varphi \leftrightarrow \neg \PR_T(\gdl{\varphi})$. 
Then we have $T \vdash \Con_{\PR_T}(\varphi) \to \varphi$ as in \cite{Kur2}. 
Since $T \nvdash \varphi$, $T \nvdash \Con_{\PR_T}(\varphi)$. 
Thus $T \nvdash \Con_{\PR_T}^S$. 
\end{proof}

Notice that if $\PR_T(x)$ satisfies $\D{2}$, then $\Con_{\PR_T}^S$ follows from $\Con_{\PR_T}^L$. 
Hence Theorem \ref{Lob} is also a consequence of Theorem \ref{KurG2}. 

In his proof of the incompleteness theorems, G\"odel constructed a $\Delta_1(\PA)$ formula ${\sf Proof}_T(x, y)$ saying that $y$ is the G\"odel number of a $T$-proof of a formula with the G\"odel number $x$. 
G\"odel's provability predicate ${\sf Prov}_T(x)$ is defined as $\exists y {\sf Proof}_T(x, y)$. 
Then the formula ${\sf Prov}_T(x)$ is a $\Sigma_1$ provability predicate satisfying full derivability conditions $\DU{1}$, $\DG{2}$ and $\SCG$. 
Thus the sentence $\Con_{{\sf Prov}_T}^L$ is not provable in $T$ by Theorem \ref{Lob}. 
This is G\"odel's second incompleteness theorem. 
It is known that $\Con_{{\sf Prov}_T}^H$ and $\Con_{{\sf Prov}_T}^L$ are provably equivalent in $\PA$ (see \cite{Kur2}), and so let $\Con_T$ denote one of these consistency statements.  

\begin{theorem}[The second incompleteness theorem (G\"odel \cite{Goed31})]\label{G2}
The consistency statement $\Con_T$ of $T$ cannot be proved in $T$. 
\end{theorem}

\section{Rosser provability predicates}\label{sec-rpp}

In this section, we introduce Rosser provability predicates and survey on derivability conditions for Rosser provability predicates. 
We say a formula $\Prf_T(x, y)$ is a {\it proof predicate} of $T$ if it satisfies the following conditions: 
\begin{enumerate}
	\item $\Prf(x, y)$ is $\Delta_1(\PA)$; 
	\item $\PA \vdash \forall x({\sf Prov}_T(x) \leftrightarrow \exists y \Prf_T(x, y))$; 
	\item for any natural number $n$ and formula $\varphi$, $\N \models {\sf Proof}_T(\gdl{\varphi}, \overline{n}) \leftrightarrow \Prf_T(\gdl{\varphi}, \overline{n})$; 
	\item $\PA \vdash \forall x \forall x' \forall y(\Prf_T(x, y) \land \Prf_T(x', y) \to x = x')$. 
\end{enumerate}

The last clause means that our proof predicates are single-conclusion ones. 

For each proof predicate $\Prf_T(x, y)$ of $T$, we can associate the $\Sigma_1$ formula
\[
	\exists y(\Prf_T(x, y) \land \forall z \leq y \neg \Prf_T(\dot{\neg} x, z))
\]
which is said to be the {\it Rosser provability predicate} of $\Prf_T(x, y)$ or a Rosser provability predicate of $T$. 
Notice that every Rosser provability predicate of $T$ is a $\Sigma_1$ provability predicate of $T$. 

Rosser provability predicates were essentially introduced by Rosser \cite{Ros36} to improve G\"odel's first incompleteness theorem. 
The following proposition is an important feature of Rosser provability predicates. 

\begin{proposition}\label{RP}
Let $\PRR_T(x)$ be any Rosser provability predicate of $T$ and $\varphi$ be any formula. 
If $T \vdash \neg \varphi$, then $\PA \vdash \neg \PRR_T(\gdl{\varphi})$. 
\end{proposition}

Since $T$ proves $\neg 0 \neq 0$, $\PA \vdash \neg \PRR_T(\gdl{0 \neq 0})$ by Proposition \ref{RP}. 
Thus we obtain the following proposition. 

\begin{proposition}\label{RP2}
For any Rosser provability predicate $\PRR_T(x)$ of $T$, $\PA \vdash \Con_{\PRR_T}^L$. 
\end{proposition}

There are limitations of the existence of Rosser provability predicates satisfying certain derivability conditions. 
From Theorem \ref{Lob} and Proposition \ref{RP2}, we obtain the following corollary. 

\begin{corollary}\label{D2D3}
There exists no Rosser provability predicate of $T$ satisfying both $\D{2}$ and $\D{3}$. 
\end{corollary}

The following proposition is implicitly stated in Jeroslow \cite{Jer73} without a proof. 

\begin{proposition}\label{RSC}
There exists no Rosser provability predicate of $T$ satisfying $\SC$. 
\end{proposition}
\begin{proof}
Let $\PRR_T(x)$ be the Rosser provability predicate of a proof predicate $\Prf_T(x, y)$. 
Let $\sigma$ be a $\Sigma_1$ sentence satisfying the following equivalence:
\[
	\PA \vdash \sigma \leftrightarrow \exists x(\Prf_T(\gdl{\neg \sigma}, x) \land \forall y < x \neg \Prf_T(\gdl{\sigma}, y)). 
\]
Then $\PA \vdash \PRR_T(\gdl{\sigma}) \to \neg \sigma$. 
If $T \vdash \sigma \to \PRR_T(\gdl{\sigma})$, 
then $T \vdash \sigma \to \neg \sigma$, and hence $T \vdash \neg \sigma$. 
Then $\N \models \sigma$ by the choice of $\sigma$. 
By $\Sigma_1$-completeness, $T \vdash \sigma$. 
This is a contradiction. 
We conclude $T \nvdash \sigma \to \PRR_T(\gdl{\sigma})$. 
\end{proof}

Then by Proposition \ref{RSC}, Theorem \ref{TK} and Corollary \ref{BucT}, we obtain the following corollary. 

\begin{corollary}\label{BK}\leavevmode
\begin{enumerate}
	\item There exists no Rosser provability predicate of $T$ satisfying $\BDU{2}$.
	\item There exists no Rosser provability predicate of $T$ satisfying both $\DU{1}$ and $\DU{2}$.
\end{enumerate}
\end{corollary}

Kreisel and Takeuti \cite{KT74} asked whether Rosser provability predicates satisfy $\D{2}$ or not. 
Guaspari and Solovay established a modal logical method of obtaining Rosser provabilty predicates without some certain conditions. 
From their method, we have:  

\begin{theorem}[Guaspari and Solovay \cite{GS79}]
There exist Rosser provability predicates satisfying neither $\D{2}$ nor $\D{3}$. 
\end{theorem}

Notice that Guaspari and Solovay's Rosser provability predicates are based on multi-conclusion proof predicates, and Shavrukov \cite{Sha91} proved the same result for Rosser provability predicates based on single-conclusion proof predicates. 

The Rosser provability predicate of G\"odel's proof predicate ${\sf Proof}_T(x, y)$ is denoted by ${\sf Prov}_T^R(x)$. 
Montagna \cite{Mon79} proved that the global version of $\D{2}$ does not hold for ${\sf Prov}_T^R(x)$. 

\begin{proposition}[Montagna \cite{Mon79}]
The Rosser provability predicate ${\sf Prov}_T^R(x)$ does not satisfy $\DG{2}$. 
\end{proposition}

On the other hand, there are Rosser provability predicates satisfying some derivability conditions. 

\begin{theorem}[Bernardi and Montagna \cite{BM84}; Arai \cite{Ara90}]\label{NER1}
There exist Rosser provability predicates satisfying $\DG{2}$. 
\end{theorem}

The existence of Rosser provability predicates satisfying $\D{2}$ was also mentioned in Shavrukov \cite{Sha91}. 
Kikuchi and Kurahashi \cite{KK17} investigated Rosser provability predicates satisfying $\D{2}$ and an additional certain condition. 
Kurahashi \cite{Kur} investigated provability logics of Rosser provability predicates satisfying $\D{2}$. 

The existence of Rosser provability predicates satisfying $\D{3}$ was proved by Arai. 

\begin{theorem}[Arai \cite{Ara90}]\label{NER2}
There exist Rosser provability predicates satisfying $\DG{3}$. 
\end{theorem}

Arai proved Theorems \ref{NER1} and \ref{NER2} under the assumption that formulas are formulated in negation normal form. 
Let ${\sf nnf}(\varphi)$ be one of negation normal forms of a formula $\varphi$ such that ${\sf nnf}(\neg \neg \varphi) \equiv {\sf nnf}(\varphi)$. 
Then we can understand that Arai's Rosser provability predicates are defined as $\PRR_T(x) \equiv \exists y(\Prf_T({\sf nnf}(x), y) \land \forall z \leq y \neg \Prf_T({\sf nnf}(\dot{\neg} x), z))$. 
Then it is easy to see that $\PA \vdash \Con_{\PRR_T}^H$. 
Thus strictly speaking, from this point of view, Arai's Rosser provability predicates are different from ours. 
Notice that his proofs are applicable to our Rosser provability predicates with some modifications. 
In fact, our proof of Theorem \ref{DCT1} (resp.~Theorem \ref{DCT3}) improves Theorem \ref{NER1} (resp.~Theorem \ref{NER2}), and our proofs of these theorems are based on Arai's proofs. 

Moreover, Arai's results indicate that neither $\{\D{1}, \DG{2}\}$ nor $\{\D{1}, \DG{3}\}$ is sufficient for the unprovability of the consistency statement $\Con_{\PR_T}^H$. 
By Corollaries \ref{D2D3} and \ref{BK}, $\{\D{1}, \DG{2}\}$ implies neither $\D{3}$ nor $\DU{1}$. 
A similar observation can be done for $\{\D{1}, \DG{3}\}$, and then we obtain the following corollary on non-implications concerning derivability conditions and the second incompleteness theorem.
 
\begin{corollary}\leavevmode
\begin{enumerate}
	\item $\{\D{1}, \DG{2}\}$ does not imply any of $\DU{1}$, $\D3$ or $T \nvdash \Con_{\PR_T}^H$. 
	\item $\{\D{1}, \DG{3}\}$ does not imply any of $\BD{2}$, $\SC$ or $T \nvdash \Con_{\PR_T}^H$. 
\end{enumerate}
\end{corollary}

For example, the first clause of this corollary means that there exists a $\Sigma_1$ provability predicate $\PR_T(x)$ satisfying $\D{1}$ and $\DG{2}$, and not enjoying any of $\DU{1}$, $\D{3}$ or $T \nvdash \Con_{\PR_T}^H$. 

In this paper, we improve Theorems \ref{NER1} and \ref{NER2} by showing the existence of Rosser provability predicates satisfying more additional conditions. 
As a corollary to our results, we obtain several more non-implications.

\section{Main Theorems}

In this section, we prove three theorems which are main theorems of this paper. 
The first theorem is an improvement of Theorem \ref{NER1}. 

\begin{theorem}\label{DCT1}
There exists a Rosser provability predicate of $T$ satisfying $\DG{2}$, $\DCG$ and $\PA \vdash \Con_{\PRR_1}^H$. 
That is, there exists a Rosser provability predicate $\PRR_1(x)$ of $T$ satisfying the following conditions: 
\begin{enumerate}
	\item $\PA \vdash \forall x \forall y(\PRR_1(x \dot{\to} y) \to (\PRR_1(x) \to \PRR_1(y)))$. 
	\item $\PA \vdash \forall x({\sf True}_{\Delta_0}(x) \to \PRR_1(x))$. 
	\item $\PA \vdash \forall x(\PRR_1(x) \to \neg \PRR_1(\dot{\neg}x))$. 
\end{enumerate}
\end{theorem}

The second theorem shows that in the statement of Corollary \ref{BucT}, the condition $\DU{2}$ cannot be replaced by $\D{2}$. 

\begin{theorem}\label{DCT2}
There exists a Rosser provability predicate of $T$ satisfying $\CB$, $\D{2}$ and $\DCG$. 
That is, there exists a Rosser provability predicate $\PRR_2(x)$ of $T$ satisfying the following conditions: 
\begin{enumerate}
	\item $T \vdash \PRR_2(\gdl{\forall \vec{x}\, \varphi(\vec{x})}) \to \forall \vec{x}\, \PRR_2(\gdl{\varphi(\vec{\dot{x}})})$ for any formula $\varphi(\vec{x})$.  
	\item $T \vdash \PRR_2(\gdl{\varphi \to \psi}) \to (\PRR_2(\gdl{\varphi}) \to \PRR_2(\gdl{\psi}))$ for any formulas $\varphi$ and $\psi$. 
	\item $\PA \vdash \forall x({\sf True}_{\Delta_0}(x) \to \PRR_2(x))$. 
\end{enumerate}
\end{theorem}

The last theorem is an improvement of Theorem \ref{NER2}. 

\begin{theorem}\label{DCT3}
There exists a Rosser provability predicate of $T$ satisfying $\CB$, $\BD{2}$, $\DG{3}$ and $\DCG$. 
That is, there exists a Rosser provability predicate $\PRR_3(x)$ of $T$ satisfying the following conditions: 
\begin{enumerate}
	\item $T \vdash \PRR_3(\gdl{\forall \vec{x}\, \varphi(\vec{x})}) \to \forall \vec{x}\, \PRR_3(\gdl{\varphi(\vec{\dot{x}})})$ for any formula $\varphi(\vec{x})$. 
	\item For any formulas $\varphi$ and $\psi$, if $T \vdash \varphi \to \psi$, then $T \vdash \PRR_3(\gdl{\varphi}) \to \PRR_3(\gdl{\psi})$. 
	\item $\PA \vdash \forall x (\PRR_3(x) \to \PRR_3(\gdl{\PRR_3(\dot{x})}))$. 
	\item $\PA \vdash \forall x({\sf True}_{\Delta_0}(x) \to \PRR_3(x))$. 
\end{enumerate}
\end{theorem}

\begin{remark}
Notice that both of $\PRR_2(x)$ and $\PRR_3(x)$ satisfy the Hilbert-Bernays derivability conditions $\BD{2}$, $\CB$ and $\DCU$. 
Thus by Theorem \ref{HB}, $T$ can prove neither $\Con_{\PRR_2}^H$ nor $\Con_{\PRR_3}^H$. 
Therefore our Theorems \ref{DCT2} and \ref{DCT3} cannot be proved under Arai's assumption that formulas are formulated in negation normal form. 
Moreover, since $\PRR_3(x)$ satisfies $\BD{2}$ and $\D{3}$, $T \nvdash \Con_{\PRR_3}^S$ by Theorem \ref{KurG2}. 

By Proposition \ref{RP2}, $T$ proves $\Con_{\PRR_2}^L$ and $\Con_{\PRR_3}^L$. 
Since $\PRR_2(x)$ satisfies $\D{2}$, $T$ can also prove $\Con_{\PRR_2}^S$. 
Therefore our Rosser provability predicates $\PRR_2(x)$ and $\PRR_3(x)$ indicate the difference of the unprovability of three consistency statements $\Con_{\PR_T}^H$, $\Con_{\PR_T}^S$ and $\Con_{\PR_T}^L$. 
The following table summarizes the situation of the unprovability of consistency statements. 

\begin{table}[h]
\begin{center}
\begin{tabular}{|c|c|c|c|}
\hline
 & $T \vdash \Con_{\PR_T}^H$ & $T \vdash \Con_{\PR_T}^S$ & $T \vdash \Con_{\PR_T}^L$ \\
\hline
$\PRR_1(x)$ & Yes & Yes & Yes \\
\hline
$\PRR_2(x)$ & No & Yes & Yes \\
\hline
$\PRR_3(x)$ & No & No & Yes \\
\hline
\end{tabular}
\end{center}
\end{table}
\end{remark}

We obtain the following corollary for provability predicates. 

\begin{corollary}\label{MTC}\leavevmode
\begin{enumerate}
	\item $\{\D{1}, \DG{2}, \DCG\}$ does not imply any of $\DU{1}$, $\D{3}$ or $T \nvdash \Con_{\PR_T}^H$. 
	\item $\{\D{1}, \CB, \D{2}, \DCG\}$ does not imply any of $\DU{2}$, $\D{3}$ or $T \nvdash \Con_{\PR_T}^S$. 
	\item $\{\D{1}, \CB, \BD{2}, \DG{3}, \DCG\}$ does not imply any of $\D{2}$, $\SC$ or $T \nvdash \Con_{\PR_T}^L$. 
\end{enumerate}
\end{corollary}

Corollary \ref{MTC} shows that $\Con_{\PR_T}^H$ and $\Con_{\PR_T}^S$ cannot be replaced by $\Con_{\PR_T}^S$ and $\Con_{\PR_T}^L$ in the statements of Theorems \ref{HB} and \ref{KurG2}, respectively. 

\begin{remark}
Since the set of all $\Delta_0$ sentences are closed under negation, $\PA \vdash \forall x(\Delta_0(x) \land \neg {\sf True}_{\Delta_0}(x) \to {\sf True}_{\Delta_0}(\dot{\neg}x))$ where $\Delta_0(x)$ is a formula naturally representing the set of all $\Delta_0$ sentences. By $\DCG$ for $\PRR_1(x)$, we have that $\PA \vdash \forall x(\Delta_0(x) \land \neg {\sf True}_{\Delta_0}(x) \to \PRR_1(\dot{\neg} x))$. 
Since $\Con_{\PRR_1}^H$ is provable in $\PA$, we obtain $\PA \vdash \forall x(\Delta_0(x) \land \neg {\sf True}_{\Delta_0}(x) \to \neg \PRR_1(x))$. 
Therefore the uniform $\Delta_0$ reflection principle $\forall x(\Delta_0(x) \land \PRR_1(x) \to {\sf True}_{\Delta_0}(x))$ for $\PRR_1(x)$ is provable in $\PA$.\footnote{This is pointed out by the referee.}
\end{remark}

Before proving our results, we introduce some terminology and prove a lemma. 
We assume that our logical symbols are only $\land, \neg$ and $\forall$, and other logical symbols such as $\to$ and $\exists$ are introduced as abbreviations. 
We say a formula $\varphi'$ is an {\it instance} of a formula $\varphi$ if for some numbers $k, i_0, \ldots, i_{k-1}$, some variables $x_0, \ldots, x_{k-1}$ and some formula $\psi$, the formulas $\varphi$ and $\varphi'$ are of the forms $\forall x_0 \cdots \forall x_{k-1} \psi(x_0, \ldots, x_{k-1})$ and $\psi(\overline{i_0}, \ldots, \overline{i_{k-1}})$, respectively. 
For each natural number $m$, let $F_m$ be the set of all formulas whose G\"odel numbers are less than or equal to $m$. 
We say that a finite mapping $V : F_m \to \{0, 1\}$ is a {\em truth assignment on $F_m$} if $V$ satisfies the usual conditions of truth assignments for propositional logic such as $V(\varphi \land \psi) = V(\varphi) \cdot V(\psi)$, $V(\neg \varphi) = 1 - V(\varphi)$, and so on. 
Let $X$ be any finite set of formulas. 
Let $d(X) = \min \{n' : X \subseteq F_{n'}\}$. 
A truth assignment $V$ on $F_m$ is said to be a {\em model of $X$} if $d(X) \leq m$ and $V(\varphi) = 1$ for all $\varphi \in X$. 
Let $P_{T, m}$ be the finite set $\{\varphi : \N \models \exists y \leq \overline{m}\, {\sf Proof}_T(\gdl{\varphi}, y)\}$ of formulas. 

We introduce two conditions (A) and (B) about truth assignments $V$ on $F_m$. 


\begin{itemize}
	\item (A): For any formulas $\varphi, \varphi' \in F_m$, if $V(\varphi) = 1$ and $\varphi'$ is an instance of $\varphi$, then $V(\varphi') = 1$.  
	\item (B): If $\varphi \in F_m$ is a $\Delta_0$ sentence, then $\varphi$ is true if and only if $V(\varphi) = 1$. 
\end{itemize}

The above terminology and definitions are formalizable in $\PA$. 
Then we can define a $\Delta_1(\PA)$ formula ${\sf Sat}(u)$ saying that ``there exists a model of $P_{T, u}$ satisfying the conditions (A) and (B)'' by using the $\Delta_1(\PA)$ formula ${\sf True}_{\Delta_0}(x)$. 
We prove the following lemma.

\begin{lemma}\label{DCL1}
$\PA \vdash \Con_T \leftrightarrow \forall u {\sf Sat}(u)$. 
\end{lemma}
\begin{proof}
We reason in $\PA$. 

$(\leftarrow)$: Suppose $\neg \Con_T$, then there exists a number $m$ such that $P_{T, m}$ contains both $0 = 0$ and $0 \neq 0$. 
Obviously $P_{T, m}$ has no model. 
Therefore $\neg {\sf Sat}(m)$ holds for some $m$. 

$(\rightarrow)$: Suppose $\Con_T$. 
Then there is a definable complete consistent extension $T'$ of $T$ by the arithmetized completeness theorem (cf.~H\'ajek and Pudl\'ak \cite{HP93}). 
Let $m$ be any number. 
We define a finite mapping $V : F_m \to \{0, 1\}$ as follows: for every $\varphi \in F_m$, $V(\varphi) = 1$ if and only if $\varphi \in T'$. 
Since $T'$ is complete and consistent, $V$ is a truth assignment on $F_m$. 

We prove that $V$ satisfies the condition (A). 
Let $\varphi$ and $\varphi'$ be any formulas in $F_m$ with $V(\varphi)=1$ and $\varphi'$ is an instance of $\varphi$. 
Then $\varphi \in T'$. 
Since $\varphi \to \varphi'$ is logically valid, $\varphi' \in T'$. 
Therefore $V(\varphi') = 1$. 

We prove that $V$ satisfies the condition (B). 
Let $\varphi \in F_m$ be any $\Delta_0$ sentence. 
If $\varphi$ is true, then $\varphi$ is provable in $T$ by $\DCG$ for ${\sf Prov}_T(x)$. 
Thus $\varphi \in T'$ and hence $V(\varphi) = 1$. 
If $\varphi$ is false, then $\neg \varphi$ is a true $\Delta_0$ sentence and is provable in $T$. 
By the consistency of $T'$, $\varphi \notin T'$. 
Therefore $V(\varphi) = 0$. 

Also since $T'$ is an extension of $T$, $V$ is a model of $P_{T, m}$. 
We conclude that ${\sf Sat}(m)$ holds. 
\end{proof}

Notice that $\PA$ is essentially reflexive, that is, every consistent extension of $\PA$ can prove the consistency of every finite subtheory of itself (cf.~\cite{Lin03}). 
Thus in a similar way as in our proof of Lemma \ref{DCL1}, we obtain the following lemma.

\begin{lemma}\label{DCL2}
For every natural number $m$, $T \vdash {\sf Sat}(\overline{m})$. 
\end{lemma}

\subsection{Proof of Theorem \ref{DCT1}}

In this subsection, we prove Theorem \ref{DCT1}. 
For each formula $\varphi$, we define the formula $-\varphi$ as follows: 
\[
	- \varphi : \equiv \begin{cases}
	\varphi & \text{if}\ \varphi\ \text{is not of the form}\ \neg \psi, \\
	\psi & \text{if}\ \varphi\ \text{is of the form}\ \neg \psi.
	\end{cases}
\]

Let $\{\xi_k\}_{k \in \omega}$ be the effective enumeration of all formulas introduced in Section \ref{sec-dc}. 
Notice that if $\xi_k$ is $- \xi_l$, then $\xi_k$ is a subformula of $\xi_l$, and hence $k \leq l$. 
Therefore $\xi_k$ is none of $-\xi_l$ for all $l < k$. 
This property will be used in our proofs. 

We define a primitive recursive function $e(\varphi, V, n)$ as follows: 
\begin{itemize}
	\item If $V$ is not a (code of) truth assignment on $F_n$, $e(\varphi, V, n) = 0$. 
	\item If $V$ is a truth assignment on $F_n$, then the value of $e(\varphi, V, n)$ is defined as follows by recursion on the construction of $\varphi$: 
	\begin{enumerate}
	\item If $\varphi$ is an atomic formula or a universal formula, 
\[
	e(\varphi, V, n) := \begin{cases}
	V(\varphi) & \text{if}\ \varphi \in F_n, \\
	1 & \text{if}\ \varphi \notin F_n\ \&\ \varphi\ \text{is a true}\ \Delta_0\ \text{sentence}, \\
	0 & \text{if}\ \varphi \notin F_n\ \&\ \varphi\ \text{is a false}\ \Delta_0\ \text{sentence}, \\
	1 & \text{otherwise}.
	\end{cases}
\]
	\item If $\varphi$ is $\neg \xi_0$, then $e(\varphi, V, n) := 1 - e(\xi_0, V, n)$. 
	\item If $\varphi$ is $\xi_0 \land \xi_1$, then $e(\varphi, V, n) := e(\xi_0, V, n) \cdot e(\xi_1, V, n)$. 
\end{enumerate}
\end{itemize}

Then it can be proved that if $V$ is a truth assignment on $F_n$, then for any formula $\varphi \in F_n$, $V(\varphi) = e(\varphi, V, n)$. 

Here we state Theorem \ref{DCT1} again.

\setcounter{theorem}{0}

\begin{theorem}
There exists a Rosser provability predicate $\PRR_1(x)$ of $T$ satisfying the following conditions: 
\begin{enumerate}
	\item $\PA \vdash \forall x \forall y(\PRR_1(x \dot{\to} y) \to (\PRR_1(x) \to \PRR_1(y)))$. 
	\item $\PA \vdash \forall x({\sf True}_{\Delta_0}(x) \to \PRR_1(x))$. 
	\item $\PA \vdash \forall x(\PRR_1(x) \to \neg \PRR_1(\dot{\neg}x))$. 
\end{enumerate}
\end{theorem}

\begin{proof}
We define a $\PA$-provably recursive function $g_1(x)$ in stages. 
Let $\Prf_1(x, y)$ be the $\Delta_1(\PA)$ formula $x = g_1(y) \land {\sf Fml}(x)$, where ${\sf Fml}(x)$ is a natural $\Delta_1(\PA)$ formula saying that $x$ is a formula. 
Let $\PRR_1(x)$ be the Rosser provability predicate of $\Prf_1(x, y)$. 
The definition of $g_1$ consists of Procedures 1 and 2. 
The definition of $g_1$ begins with Procedure 1, and enumerates theorems of $T$ until appearing a number $m$ such that ${\sf Sat}(m)$ does not hold. 
After appearing such a number $m$, the definition of $g_1$ shifts to Procedure 2. 
In Procedure 2, $g_1$ outputs all formulas in stages. 
In the definition of the function $g_1$, we identify each formula with its G\"odel number. 

\vspace{0.1in}

Procedure 1. \\
Stage $1.m$: 
\begin{itemize}
		\item If ${\sf Sat}(m)$, then
\[
	g_1(m) = \begin{cases}  \varphi & \text{if}\ m\ \text{is a proof of}\ \varphi\ \text{in} \ T,\ \text{that is},\ {\sf Proof}_T(\varphi, m)\ \text{holds}, \\
			0 & m\ \text{is not a proof of any formula in}\ T.
	\end{cases}
\]
	Go to Stage $1.(m+1)$. 
	\item If $\neg {\sf Sat}(m)$, then go to Procedure 2. 
\end{itemize}

Procedure 2. \\
Let $m$ be the least number such that ${\sf Sat}(m)$ does not hold. 
Since ${\sf Sat}(m-1)$ holds, there exists a model of $P_{T, m-1}$ satisfying the condition (B). 
Let $V$ be the least such model. 
Let $n = d(P_{T, m - 1})$, then $V$ is defined on $F_n$. 

We define the value $g_1(m + k)$ for $k \geq 0$ as follows: 
\[
	g_1(m + k) = \begin{cases}
	- \xi_k & \text{if}\ e(\xi_k, V, n) = 1, \\
	\neg \xi_k & \text{if}\ e(\xi_k, V, n) = 0. 
	\end{cases}
\]
The definition of the function $g_1$ is completed. 

First, we show that the formula ${\rm Prf}_1(x, y)$ is a proof predicate of $T$. 
Since $\forall u{\sf Sat}(u)$ is true in the standard model $\N$ by Lemma \ref{DCL1}, $\N \models {\sf Proof}_T(\gdl{\varphi}, \overline{n}) \leftrightarrow {\rm Prf}_1(\gdl{\varphi}, \overline{n})$ for all $\varphi$ and $n$ by the definition of $g_1$. 
Also $\PA \vdash \forall x \forall x' \forall y({\rm Prf}_1(x, y) \land {\rm Prf}_1(x', y) \to x = x')$ trivially holds. 
We show that the sentence $\forall x({\sf Prov}_T(x) \leftrightarrow \exists y \Prf_1(x, y))$ is provable in $\PA$. 
By the definition of $g_1$, this sentence is obviously proved in $\PA + \forall u {\sf Sat}(u)$. 
It suffices to show that this sentence is provable in $\PA + \exists u \neg {\sf Sat}(u)$. 

We proceed in $\PA + \exists u \neg {\sf Sat}(u)$: 
Let $m$ be the least number such that ${\sf Sat}(m)$ does not hold, and let $n = d(P_{T, m - 1})$. 
Let $V$ be the least model of $P_{T, m - 1}$ satisfying (B). 
We show that for any $k$, $g_1$ eventually outputs the formula $\xi_k$. 
We distinguish the following three cases. 
\begin{itemize}
	\item $e(\xi_k, V, n) = 1$ and $\xi_k$ is not of the form $\neg \psi$: Then $g_1(m + k) = - \xi_k = \xi_k$.	
	\item $e(\xi_k, V, n) = 1$ and $\xi_k$ is of the form $\neg \xi_l$: Then $e(\xi_l, V, n) = 0$ and $g_1(m + l) = \neg \xi_l = \xi_k$. 
	\item $e(\xi_k, V, n) = 0$: Then for $p$ with $\neg \xi_k = \xi_p$, $e(\xi_p, V, n) = 1$ and hence $g_1(m + p) = - \xi_p = \xi_k$. 
\end{itemize}
We have shown that $\exists y {\rm Prf}_1(x, y)$ holds if and only if $x$ is a formula. 
Since $\neg \Con_T$ holds by Lemma \ref{DCL1}, $x$ is provable in $T$ if and only if $x$ is a formula. 
Therefore $\forall x({\sf Prov}_T(x) \leftrightarrow \exists y {\rm Prf}_1(x, y))$ holds.  

Next, we show that our formula $\PRR_1(x)$ satisfies the required conditions. 
The following claim is a key property of our construction of $\PRR_1(x)$. 

\vspace{0.1in}
{\bf Claim 1.}
The following sentence is provable in $\PA$: 
``Let $m$ be the least number such that ${\sf Sat}(m)$ does not hold, let $n = d(P_{T, m - 1})$ and let $V$ be the least model of $P_{T, m - 1}$ satisfying (B). 
Then for any formula $\varphi$, 
\[
	e(\varphi, V, n) = 1\ \text{if and only if}\ \PRR_1(\gdl{\varphi})\ \text{holds''}.
\]
\begin{proof}
We proceed in $\PA$. 
Let $m$, $n$ and $V$ be as indicated in the statement. 
Let $\varphi$ be any formula. 

$(\Rightarrow)$: 
Suppose $e(\varphi, V, n) = 1$. 
If $\neg \varphi \in F_n$, then $V(\neg \varphi) = 1 - V(\varphi) = 1 - e(\varphi, V, n) = 0$. 
Therefore $\neg \varphi$ is not in $P_{T, m - 1}$ because $V$ is a model of $P_{T, m - 1}$. 
If $\neg \varphi \notin F_n$, then $\neg \varphi$ is not in $P_{T, m - 1}$ because $n = d(P_{T, m - 1})$. 
In either case, $\neg \varphi$ is not in $P_{T, m - 1}$. 
Hence $\neg \varphi$ is not in $\{g_1(0), \ldots, g_1(m-1)\}$ because the construction of $g_1$ executes Procedure 1 before Stage $1.m$. 

Let $\varphi = \xi_k$. 
Then $\neg \varphi$ is not in $\{g_1(m), \ldots, g_1(m+k)\}$ because $\neg \varphi$ is neither $- \xi_{k'}$ nor $\neg \xi_{k''}$ for all $k' \leq k$ and $k'' < k$.  
If $\varphi$ is not of the form $\neg \psi$, then $g_1(m + k) = - \varphi = \varphi$ because $e(\xi_k, V, n) = 1$. 
If $\varphi$ is of the form $\neg \xi_l$, then $l < k$ and $g_1(m + l) = \neg \xi_l = \varphi$ because $e(\xi_l, V, n) = 0$. 
Therefore $\PRR_1(\gdl{\varphi})$ holds in either case. 

$(\Leftarrow)$: Suppose $e(\varphi, V, n) = 0$. 
Then $\varphi \notin \{g_1(0), \ldots, g_1(m-1)\}$ because $n = d(P_{T, m - 1})$ and $V$ is a model of $P_{T, m-1}$. 
Let $\varphi = \xi_k$, then $g_1(m + k) = \neg \xi_k = \neg \varphi$. 
If $\varphi$ is not of the form $\neg \psi$, then $\varphi \notin \{g_1(m), \ldots, g_1(m + k -1)\}$ because $\varphi$ is neither $- \xi_{k'}$ nor $\neg \xi_{k'}$ for all $k' < k$. 
If $\varphi$ is of the form $\neg \xi_l$ for some $l < k$, then $\varphi$ is neither $- \xi_{l'}$ nor $\neg \xi_{l'}$ for all $l' < l$. 
Since $e(\xi_l, V, n) = 1 - e(\varphi, V, n) = 1$, $g_1(m + l) = -\xi_l \neq \varphi$. 
Hence we have $\varphi \notin \{g_1(m), \ldots, g_1(m + k -1)\}$. 
In either case, $\PRR_1(\gdl{\varphi})$ does not hold. 
\end{proof}

\vspace{0.1in}
{\bf Claim 2.}
$\PA \vdash \forall x \forall y(\PRR_1(x \dot{\to} y) \to (\PRR_1(x) \to \PRR_1(y)))$. 
\begin{proof}
Since $\PA + \forall u {\sf Sat}(u) \vdash \Con_T$ by Lemma \ref{DCL1}, $\PA + \forall u {\sf Sat}(u) \vdash \forall x ({\sf Prov}_T(x) \leftrightarrow \PRR_1(x))$ easily follows from the definition of $g_1$. 
Then $\DG{2}$ for ${\sf Prov}_T(x)$ implies $\PA + \forall u {\sf Sat}(u) \vdash \forall x \forall y(\PRR_1(x \dot{\to} y) \to (\PRR_1(x) \to \PRR_1(y)))$. 

We proceed in $\PA + \exists u \neg {\sf Sat}(u)$: 
Let $n$ and $V$ be as above. 
Let $\varphi$ and $\psi$ be any formulas with $\PRR_1(\gdl{\varphi \to \psi})$ and $\PRR_1(\gdl{\varphi})$ hold. 
Then $e(\varphi \to \psi, V, n) = e(\varphi, V, n) = 1$ by Claim 1. 
We have $e(\psi, V, n) = 1$ and hence $\PRR_1(\gdl{\psi})$ holds by Claim 1 again. 
We have shown that the theory $\PA + \exists u \neg {\sf Sat}(u)$ also proves $\forall x \forall y(\PRR_1(x \dot{\to} y) \to (\PRR_1(x) \to \PRR_1(y)))$. 
\end{proof}

\vspace{0.1in}
{\bf Claim 3.}
$\PA \vdash \forall x ({\sf True}_{\Delta_0}(x) \to \PRR_1(x))$. 
\begin{proof}
As in the proof of Claim 2, $\DCG$ for ${\sf Prov}_T(x)$ implies $\PA + \forall u {\sf Sat}(u) \vdash \forall x ({\sf True}_{\Delta_0}(x) \to \PRR_1(x))$. 

We work in $\PA + \exists u \neg {\sf Sat}(u)$: 
Let $n$ and $V$ be as above. 
First, we prove by induction on the construction of $\varphi \in \Delta_0$ that for all $\Delta_0$ sentences $\varphi$, $\varphi$ is true if and only if $e(\varphi, V, n) = 1$. 
\begin{itemize}
	\item (Base Case): $\varphi$ is an atomic sentence or a universal sentence: \\
	$(\Rightarrow)$: Suppose that $\varphi$ is true. 
	If $\varphi \notin F_n$, then $e(\varphi, V, n) = 1$ by the definition of $e$. 
	If $\varphi \in F_n$, then $V(\varphi) = 1$ by the condition (B). 
	Hence $e(\varphi, V, n) = 1$. 
	
	The proof for $(\Leftarrow)$ is similar. 
	
	\item Induction cases are straightforward by the definition of $e$. 
\end{itemize}

Let $\varphi$ be any true $\Delta_0$ sentence. 
Then $e(\varphi, V, n) = 1$ as shown above. 
By Claim 1, $\PRR_1(\gdl{\varphi})$ holds. 
We have shown $\PA + \exists u \neg {\sf Sat}(u) \vdash \forall x ({\sf True}_{\Delta_0}(x) \to \PRR_1(x))$. 
\end{proof}

\vspace{0.1in}
{\bf Claim 4.}
$\PA \vdash \forall x (\PRR_1(x) \to \neg \PRR_1(\dot{\neg}x))$. 
\begin{proof}
Since $\Con_T$ is equivalent to $\forall x(\PR_T(x) \to \neg \PR_T(\dot{\neg}x))$, $\PA + \forall u {\sf Sat}(u) \vdash \forall x(\PRR_1(x) \to \neg \PRR_1(\dot{\neg}x))$ by Lemma \ref{DCL1}.

We reason in $\PA + \exists u \neg {\sf Sat}(u)$: 
Let $n$ and $V$ be as above. 
Suppose $\PRR_1(\gdl{\varphi})$ holds for a formula $\varphi$. 
Then $e(\varphi, V, n) = 1$ by Claim 1. 
Since $e(\neg \varphi, V, n) = 1 - e(\varphi, V, n) = 0$, $\PRR_1(\gdl{\neg \varphi})$ does not hold by Claim 1 again.  
\end{proof}

This completes our proof of Theorem \ref{DCT1}. 
\end{proof}

\subsection{Proof of Theorem \ref{DCT2}}

In this subsection, we prove Theorem \ref{DCT2}. 

\begin{theorem}
There exists a Rosser provability predicate $\PRR_2(x)$ of $T$ satisfying the following conditions: 
\begin{enumerate}
	\item $T \vdash \PRR_2(\gdl{\forall \vec{x}\, \varphi(\vec{x})}) \to \forall \vec{x}\, \PRR_2(\gdl{\varphi(\vec{\dot{x}})})$ for any formula $\varphi(\vec{x})$. 
	\item $T \vdash \PRR_2(\gdl{\varphi \to \psi}) \to (\PRR_2(\gdl{\varphi}) \to \PRR_2(\gdl{\psi}))$ for any formulas $\varphi$ and $\psi$. 
	\item $\PA \vdash \forall x({\sf True}_{\Delta_0}(x) \to \PRR_2(x))$. 
\end{enumerate}
\end{theorem}

\begin{proof}
The definition of our $\PA$-provably recursive function $g_2$ corresponding to this theorem is different from the definition of the function $g_1$ in our proof of Theorem \ref{DCT1} only for Procedure 2. 
We describe only Procedure 2 of the definition of the function $g_2$.

\vspace{0.1in}
Procedure 2. ${\sf Sat}(m-1)$ holds but ${\sf Sat}(m)$ does not hold. \\
Let $n = d(P_{T, m - 1})$ and let $V$ be the least model of $P_{T, m - 1}$ satisfying the conditions (A) and (B).

We say a formula $\varphi$ is {\em critical} if $\varphi$ satisfies one of the following conditions: 
\begin{enumerate}
	\item $\varphi \in F_n$ and $V(\varphi) = 1$; 
	\item $\varphi \notin F_n$ and $\varphi$ is a true $\Delta_0$ sentence;  
	\item $\varphi \notin F_n$ and there exists a formula $\psi \in F_n$ such that $\varphi$ is an instance of $\psi$ and $V(\psi) = 1$. 
\end{enumerate}
Notice that if $\varphi \in F_n$, then $\varphi$ is critical if and only if $V(\varphi) = 1$. 

Let $\{\xi_k\}_{k \in \omega}$ be the effective enumeration of all formulas as above. 
We simultaneously define the values $g_2(m + k)$ for $k \geq 0$ and a sequence $\{i_k\}_{k \in \omega}$ of numbers as follows: Let $i_0 = 0$. 
\begin{enumerate}
	\item If $\xi_k$ is not critical and $\neg \xi_k$ is critical, then let $g_2(m + i_k) = \neg \xi_k$, $g_2(m + i_k + 1) = \xi_k$ and $i_{k+1} = i_k + 2$. 
	\item Otherwise, let $g_2(m + i_k) = \xi_k$ and $i_{k+1} = i_k + 1$. 
\end{enumerate}

The definition of $g_2$ is finished. 
As in the proof of Theorem \ref{DCT1}, it can be shown that ${\rm Prf}_2(x, y)$ is a proof predicate of $T$, and we omit the proof. 

\vspace{0.1in}
{\bf Claim 1.}
The following sentence is provable in $\PA$: 
``Let $m$ be the least number such that ${\sf Sat}(m)$ does not hold, let $n$ be $d(P_{T, m-1})$ and let $V$ be the least model of $P_{T, m - 1}$ satisfying the conditions (A) and (B). 
Then for any formula $\varphi$, 
\begin{enumerate}
	\item if $\varphi$ is critical, then $\PRR_2(\gdl{\varphi})$ holds, 
	\item and if $\neg \varphi \in F_n$ and $\neg \varphi$ is critical, then $\PRR_2(\gdl{\varphi})$ does not hold''.
\end{enumerate}
\begin{proof}
We procced in $\PA$. 
Let $m$, $n$ and $V$ be as indicated in the statement. 
Let $\varphi$ be any formula. 
Then for some $k$, $\varphi$ is $\xi_k$. 

1. Suppose $\xi_k$ is critical. 
First, we prove that $\neg \xi_k$ is not in $\{g_2(0), \ldots, g_2(m-1)\}$. 
If $\xi_k \in F_n$, then $V(\xi_k) = 1$, and thus $\neg \xi_k$ is not in $P_{T, m - 1}$ because $V$ is a model of $P_{T, m - 1}$. 
If $\xi_k \notin F_n$, then $\neg \xi_k \notin P_{T, m - 1}$ because $n = d(P_{T, m -1})$. 
In either case, $\neg \xi_k \notin P_{T, m - 1}$. 
Therefore $\neg \xi_k \notin \{g_2(0), \ldots, g_2(m-1)\}$. 
By the definition of $g_2$, $\neg \xi_k$ is also not in $\{g_2(m), \ldots, g_2(m+ i_k)\}$. 
Since $\xi_k$ is critical, $g_2(m + i_k) = \xi_k$. 
Hence $\PRR_2(\gdl{\xi_k})$ holds. 

2. Suppose $\neg \xi_k \in F_n$ and $\neg \xi_k$ is critical. 
Then $V(\neg \xi_k) = 1$. 
Thus $\xi_k \in F_n$ and $V(\xi_k) = 0$. 
It follows that $\xi_k$ is not critical. 
Also $\xi_k$ is not in $P_{T, m - 1}$ nor $\{g_2(0), \ldots, g_2(m-1)\}$. 
Even if $\xi_k$ is of the form $\neg \psi$, $\xi_k$ is not in $\{g_2(m), \ldots, g_2(m + i_k -1)\}$ because $\xi_k$ is not critical. 
Since $g_2(m + i_k) = \neg \xi_k$, $\PRR_2(\gdl{\xi_k})$ does not hold. 
\end{proof}

\vspace{0.1in}
{\bf Claim 2.}
$T \vdash \PRR_2(\gdl{\forall \vec{x}\, \varphi(\vec{x})}) \to \forall \vec{x}\, \PRR_2(\gdl{\varphi(\vec{\dot{x}})})$ for any formula $\varphi(\vec{x})$. 
\begin{proof}
Let $\varphi(\vec{x})$ be any formula. 
As in our proof of Theorem \ref{DCT1}, it suffices to show that the sentence is provable in $T + \exists u \neg {\sf Sat}(u)$. 
We reason in $T + \exists u \neg {\sf Sat}(u)$. 
Let $n$ and $V$ be as above. 
Notice that $n$ is larger than the G\"odel number of the formula $\neg \forall \vec{x} \varphi(\vec{x})$ by Lemma \ref{DCL2}. 
Suppose $\PRR_2(\gdl{\forall \vec{x} \varphi(\vec{x})})$ holds. 
Since $\neg \forall \vec{x} \varphi(\vec{x}) \in F_n$, $\neg \forall \vec{x} \varphi(\vec{x})$ is not critical by Claim 1. 
Then $V(\neg \forall \vec{x} \varphi(\vec{x})) \neq 1$, and hence $V(\forall \vec{x} \varphi(\vec{x})) = 1$. 
Let $\varphi(\vec{\overline{a}})$ be any instance of $\forall \vec{x} \varphi(\vec{x})$. 
If $\varphi(\vec{\overline{a}}) \in F_n$, then $V(\varphi(\vec{\overline{a}})) = 1$ by the condition (A). 
Thus $\varphi(\vec{\overline{a}})$ is critical. 
If $\varphi(\vec{\overline{a}}) \notin F_n$, then $\varphi(\vec{\overline{a}})$ is also critical because $V(\forall \vec{x} \varphi(\vec{x})) = 1$. 
In either case, $\varphi(\vec{\overline{a}})$ is critical. 
Then $\PRR_2(\gdl{\varphi(\vec{\overline{a}})})$ holds by Claim 1. 
\end{proof}

\vspace{0.1in}
{\bf Claim 3.}
$T \vdash \PRR_2(\gdl{\varphi \to \psi}) \to (\PRR_2(\gdl{\varphi}) \to \PRR_2(\gdl{\psi}))$ for any formulas $\varphi$ and $\psi$. 
\begin{proof}
Let $\varphi$ and $\psi$ be any formulas. 
We work in $T + \exists u \neg {\sf Sat}(u)$. 
Let $n$ and $V$ be as above. 
Notice that $n$ is larger than the G\"odel numbers of the formulas $\neg \varphi$ and $\neg (\varphi \to \psi)$ by Lemma \ref{DCL2}. 
Suppose $\PRR_2(\gdl{\varphi \to \psi})$ and $\PRR_2(\gdl{\varphi})$ hold. 
Since $\neg (\varphi \to \psi)$ and $\neg \varphi$ are in $F_n$, these sentences are not critical by Claim 1. 
Then $V(\neg (\varphi \to \psi)) \neq 1$ and $V(\neg \varphi) \neq 1$. 
Thus $V(\varphi \to \psi) = V(\varphi) = 1$. 
Hence $V(\psi) = 1$ and $\psi$ is critical. 
Therefore $\PRR_2(\gdl{\psi})$ holds by Claim 1. 
\end{proof}

\vspace{0.1in}
{\bf Claim 4.}
$\PA \vdash \forall x ({\sf True}_{\Delta_0}(x) \to \PRR_2(x))$. 
\begin{proof}
We proceed in $\PA + \exists u \neg {\sf Sat}(u)$. 
Let $n$ and $V$ be as above. 
Let $\varphi$ be any true $\Delta_0$ sentence. 
If $\varphi \in F_n$, then $V(\varphi) = 1$ by the condition (B). 
Thus $\varphi$ is critical. 
If $\varphi \notin F_n$, then $\varphi$ is critical because $\varphi$ is a true $\Delta_0$ sentence. 
In either case, $\varphi$ is critical. 
Therefore we have $\PRR_2(\gdl{\varphi})$ by Claim 1. 
\end{proof}

This completes our proof of Theorem \ref{DCT2}. 
\end{proof}

\subsection{Proof of Theorem \ref{DCT3}}

In this subsection, we prove Theorem \ref{DCT3}. 
Before proving the theorem, for each natural number $m$, we recursively define the sequence $\{X_{m,n}\}_{n \in \omega}$ of finite sets of negated formulas as follows: 
\begin{enumerate}
	\item $X_{m, 0} : = \{\neg \varphi : \neg \varphi \in P_{T, m}\}$. 
	\item $\neg \varphi \in X_{m, n+1}$ if and only if at least one of the following conditions holds:
	\begin{itemize}
		\item $\neg \varphi \in F_m$ and for some instance $\varphi'$ of $\varphi$, $\neg \varphi' \in X_{m, n}$. 
		\item There is a formula $\psi$ such that $\neg \psi \in X_{m, n}$ and $\neg \psi \to \neg \varphi \in P_{T, m}$. 
	\end{itemize}
\end{enumerate}
Let $X_m : = \bigcup_{n \in \omega} X_{m, n}$. 

\setcounter{theorem}{8}

\begin{lemma}\label{XYC}
Let $\varphi$ be any formula and $m$ be any natural number. 
\begin{enumerate}
	\item If $\varphi \in X_m$, then $\varphi$ is provable in $T$. 
	\item $X_m \subseteq X_{m+1}$.
	\item $X_m \subseteq F_m$. 
	As a consequence, $X_m = \bigcup_{n \leq |F_m|} X_{m, n}$ where $|F_m|$ is the number of elements of the finite set $F_m$. 
\end{enumerate}
\end{lemma}
\begin{proof}
1. We prove by induction on $n$ that for any $n \in \omega$, if $\neg \varphi \in X_{m, n}$, then $\neg \varphi$ is provable in $T$. 
\begin{itemize}
	\item If $\neg \varphi \in X_{m, 0}$, then $\neg \varphi \in P_{T, m}$ and hence $\neg \varphi$ is provable in $T$. 
	\item Assume that the statement is true for $n$. 
	Suppose $\neg \varphi \in X_{m, n+1}$. 
	If $\neg \varphi' \in X_{m, n}$ for some instance $\varphi'$ of $\varphi$, then $\neg \varphi'$ is $T$-provable by induction hypothesis. 
	Since $\varphi \to \varphi'$ is logically valid, $\neg \varphi$ is also provable in $T$. 
	
	If there is a formula $\psi$ such that $\neg \psi \in X_{m, n}$ and $\neg \psi \to \neg \varphi \in P_{T, m}$, then $\neg \psi$ is provable in $T$ by induction hypothesis. 
Since $\neg \psi \to \neg \varphi$ is $T$-provable, $\neg \varphi$ is also $T$-provable. 
\end{itemize}

2. We prove $X_{m, n} \subseteq X_{m+1, n}$ for all $n \in \omega$ by induction on $n$. 
The $n = 0$ case is immediate from $P_{T, m} \subseteq P_{T, m+1}$. 
Assume $X_{m, n} \subseteq X_{m+1, n}$. 
Suppose $\neg \varphi \in X_{m, n+1}$. 
If $\neg \varphi \in F_m$ and $\neg \varphi' \in X_{m, n}$ for some instance $\varphi'$ of $\varphi$, then $\neg \varphi \in F_{m+1}$ and $\neg \varphi' \in X_{m+1, n}$, and hence $\neg \varphi \in X_{m+1, n+1}$. 

If there is a formula $\psi$ such that $\neg \psi \in X_{m, n}$ and $\neg \psi \to \neg \varphi \in P_{T, m}$, then $\neg \varphi \in X_{m+1, n+1}$ because $\neg \psi \in X_{m+1, n}$ and $\neg \psi \to \neg \varphi \in P_{T, m+1}$. 

3. This is proved by induction and by using the fact $P_{T, m} \subseteq F_m$. 
\end{proof}

Notice that Lemma \ref{XYC} is formalizable in $\PA$. 
Also notice that there is a $\PA$-provably recursive computation calculating $X_m$ from $m$. 

Here we give our proof of Theorem \ref{DCT3}. 

\setcounter{theorem}{2}

\begin{theorem}
There exists a Rosser provability predicate $\PRR_3(x)$ of $T$ satisfying the following conditions: 
\begin{enumerate}
	\item $T \vdash \PRR_3(\gdl{\forall \vec{x}\, \varphi(\vec{x})}) \to \forall \vec{x}\, \PRR_3(\gdl{\varphi(\vec{\dot{x}})})$ for any formula $\varphi(\vec{x})$. 
	\item For any formulas $\varphi$ and $\psi$, if $T \vdash \varphi \to \psi$, then $T \vdash \PRR_3(\gdl{\varphi}) \to \PRR_3(\gdl{\psi})$. 
	\item $\PA \vdash \forall x (\PRR_3(x) \to \PRR_3(\gdl{\PRR_3(\dot{x})}))$. 
	\item $\PA \vdash \forall x ({\sf True}_{\Delta_0}(x) \to \PRR_3(x))$. 
\end{enumerate}
\end{theorem}

\begin{proof}
We define a $\PA$-provably recursive function $g_3$ corresponding to this theorem in stages. 
In the definition, as in Guaspari and Solovay \cite{GS79}, the bell which plays a role of a flag is prepared. 
As in our proofs of Theorems \ref{DCT1} and \ref{DCT2}, the construction of $g_3$ consists of Procedures 1 and 2, and the bell may ring during the execution of Procedure 1. 
When the bell rings, the construction switches to Procedure 2.  
Also in the definition of the function $g_3$, we can use the formula $\PRR_3(x)$ by the recursion theorem. 

\vspace{0.1in}
Procedure 1. \\
Stage $1.m$: 
\begin{itemize}
	\item If there exists some formula $\varphi$ satisfying at least one of the following conditions, then ring the bell and go to Procedure 2: 
	\begin{enumerate}
		\item $X_m \cup P_{T, m}$ contains both $\varphi$ and $\neg \varphi$; 
		\item $\neg \varphi \notin X_m$ and $\neg \PRR_3(\gdl{\varphi}) \in X_m$; 
		\item $\varphi$ is a true $\Delta_0$ sentence and $\neg \varphi \in X_m$. 
	\end{enumerate}

	\item Otherwise, 
\[
	g_3(m) = \begin{cases}  \varphi & \text{if}\ m\ \text{is a proof of}\ \varphi\ \text{in} \ T, \\
			0 & m\ \text{is not a proof of any formula in}\ T.
	\end{cases}
\]
	Go to Stage $1.(m+1)$. 
\end{itemize}

Procedure 2. \\
The bell rings at Stage $1.m$. 
Let $\chi_0, \ldots, \chi_{k-1}$ be a list of all elements of the finite set $X_{m-1}$. 
For $i < k$, let 
\[
	g_3(m + i) = \chi_i. 
\]

After that, let $\{\xi_i\}_{i \in \omega}$ be the effective enumeration of all formulas introduced in Section \ref{sec-dc}. 
For $i \geq 0$, let
\[
	g_3(m + k + i) = \xi_i. 
\]

Our definition of the function $g_3$ has just been completed. 
The following claim shows up an important feature of the construction of the function $g_3$. 

\vspace{0.1in}
{\bf Claim 1.}
The following statement is provable in $\PA$: 
``If the bell rings at Stage $1.m$, then for any formula $\varphi$, 
\begin{eqnarray}\label{eq1}
	\neg \varphi \in X_{m-1}\ \text{if and only if}\ \neg \PRR_3(\gdl{\varphi})\ \text{holds''}. 
\end{eqnarray}
\begin{proof}
We work in $\PA$: 
Suppose that the bell rings at Stage $1.m$. 

$(\Rightarrow)$: Suppose $\neg \varphi \in X_{m-1}$. 
Then $\neg \varphi$ is $\chi_i$ for some $i < k$ where $k = |X_{m-1}|$. 
Also $g_3(m + i) = \neg \varphi$ by the definition of $g_3$. 
If $\varphi$ were in $X_{m-1} \cup P_{T, m-1}$, then the bell would ring before Stage $1.m$, and this contradicts the choice of $m$. 
Thus $\varphi \notin X_{m-1} \cup P_{T, m-1}$. 
Then $\varphi$ is not in the list $g_3(0), \ldots, g_3(m-1), g_3(m), \ldots, g_3(m + k-1)$. 
Therefore $\neg \PRR_3(\gdl{\varphi})$ holds. 

$(\Leftarrow)$: We prove the contrapositive. 
Suppose $\neg \varphi \notin X_{m-1}$. 
In particular, $\neg \varphi \notin X_{m-1, 0}$ and hence $\neg \varphi \notin P_{T, m-1}$. 
Let $i$ and $j$ be such that $\xi_i$ is $\varphi$ and $\xi_j$ is $\neg \varphi$. 
Then $g_3(m + k + i) = \varphi$ and $g_3(m + k + j) = \neg \varphi$. 
Since $\neg \varphi \notin X_{m-1} \cup P_{T, m-1}$, $\neg \varphi$ does not appear in $g_3(0), \ldots, g_3(m + k + j -1)$. 
Since $\xi_i$ is a proper subformula of $\xi_j$, we have $i < j$. 
It follows that $\PRR_3(\gdl{\varphi})$ holds. 
\end{proof}

\vspace{0.1in}
{\bf Claim 2.}
$\PA \vdash$``the bell rings'' $\leftrightarrow \neg \Con_T$. 

\begin{proof}
We reason in $\PA$: 

$(\rightarrow)$: 
Suppose that the bell rings at Stage $1.m$. 
We distinguish the following three cases. 
\begin{itemize}
	\item $X_m \cup P_{T, m}$ contains both $\varphi$ and $\neg \varphi$: By Lemma \ref{XYC}, both $\varphi$ and $\neg \varphi$ are provable in $T$. 
Then $T$ is inconsistent. 

	\item $\neg \varphi \notin X_m$ and $\neg \PRR_3(\gdl{\varphi}) \in X_m$: 
By Lemma \ref{XYC}, $\neg \PRR_3(\gdl{\varphi})$ is provable in $T$. 
On the other hand, since $\neg \varphi \notin X_{m-1}$ by Lemma \ref{XYC}, $\PRR_3(\gdl{\varphi})$ holds by Claim 1. 
Then $\PRR_3(\gdl{\varphi})$ is provable because it is a true $\Sigma_1$ sentence. 
Therefore $T$ is inconsistent. 

	\item $\varphi$ is a true $\Delta_0$ sentence and $\neg \varphi \in X_m$: 
Then $\varphi$ is provable in $T$ by $\DCG$ for ${\sf Prov}_T(x)$. 
Also by Lemma \ref{XYC}, $\neg \varphi$ is provable in $T$. 
Hence $T$ is inconsistent. 
\end{itemize}

$(\leftarrow)$: If $T$ is inconsistent, then for some $m$ and $\varphi$, $P_{T, m}$ contains both $\varphi$ and $\neg \varphi$. 
Then the bell rings at some stage. 
\end{proof}

\vspace{0.1in}
{\bf Claim 3.}
For any natural number $n$, $T$ proves ``If the bell rings at Stage $1.m$, then $m$ is larger than $\overline{n}$''. 

\begin{proof}
Let $n$ be any natural number. 
We discuss in $T$. 
Suppose that the bell rings at Stage $1.m$ for some $m \leq n$. 
Then from our proof of Claim 2, there are formulas $\varphi$, $\neg \varphi \in F_n$ such that both $\varphi$ and $\neg \varphi$ are $T$-provable. 
Thus $P_{T, n}$ is inconsistent. 
This contradicts the reflexiveness of $T$. 
\end{proof}

Our formula ${\rm Prf}_3(x, y)$ is a proof predicate of $T$.

\vspace{0.1in}
{\bf Claim 4.}\leavevmode
\begin{enumerate}
	\item $\PA \vdash \forall x({\sf Prov}_T(x) \leftrightarrow \exists y {\rm Prf}_3(x, y))$. 
	\item For any $n \in \omega$ and formula $\varphi$, $\N \models {\sf Proof}_T(\gdl{\varphi}, \overline{n}) \leftrightarrow {\rm Prf}_3(\gdl{\varphi}, \overline{n})$. 
\end{enumerate}
\begin{proof}
1. In the theory $\PA + \neg$``the bell rings'', $\forall x({\sf Prov}_T(x) \leftrightarrow \exists y {\rm Prf}_3(x, y))$ holds by the definition of $g_3$. 

In $\PA +$``the bell rings'', $g_3$ outputs all formulas, and $T$ proves all formulas by Claim 2. 
Thus $\forall x({\sf Prov}_T(x) \leftrightarrow \exists y {\rm Prf}_3(x, y))$ also holds. 

2. Since $\N \models \neg$``the bell rings'' by Claim 2, we obtain $\N \models {\sf Proof}_T(\gdl{\varphi}, \overline{n}) \leftrightarrow {\rm Prf}_3(\gdl{\varphi}, \overline{n})$ holds for any $n \in \omega$ and formula $\varphi$ by the definition of $g_3$. 
\end{proof}

\vspace{0.1in}
{\bf Claim 5.}
$\PA \vdash \forall x(\PRR_3(x) \to \PRR_3(\gdl{\PRR_3(\dot{x})}))$. 

\begin{proof}
We work in $\PA$: 
First, suppose that the bell never rings. 
Then $T$ is consistent by Claim 2. 
Assume that $\PRR_3(\gdl{\varphi})$ holds. 
Since $\PRR_3(\gdl{\varphi})$ is a $\Sigma_1$ sentence, $\PRR_3(\gdl{\varphi})$ is provable in $T$, and hence $\PRR_3(\gdl{\varphi}) \in P_{T, m}$ for some $m$. 
By the consistency of $T$, we have $\neg \PRR_3(\gdl{\varphi}) \notin P_{T, m}$. 
Therefore $\PRR_3(\gdl{\PRR_3(\gdl{\varphi})})$ holds by the definition of $g_3$. 
We have proved that $\neg$``the bell rings'' implies $\forall x(\PRR_3(x) \to \PRR_3(\gdl{\PRR_3(\dot{x})}))$. 

Secondly, we assume that the bell rings at Stage $1.m$.
Suppose $\neg \PRR_3(\gdl{\PRR_3(\gdl{\varphi})})$ holds. 
By Claim 1, $\neg \PRR_3(\gdl{\varphi}) \in X_{m-1}$. 
If $\neg \varphi \notin X_{m-1}$, then the bell rings before Stage $1.m$. 
This is a contradiction. 
Thus $\neg \varphi \in X_{m-1}$. 
By Claim 1 again, $\neg \PRR_3(\gdl{\varphi})$ holds. 

We have proved that $\forall x(\PRR_3(x) \to \PRR_3(\gdl{\PRR_3(\dot{x})}))$ is also implied by the assumption ``the bell rings''. 
Thus we conclude that $\forall x(\PRR_3(x) \to \PRR_3(\gdl{\PRR_3(\dot{x})}))$ holds. 
\end{proof}

\vspace{0.1in}
{\bf Claim 6.}
$T \vdash \PRR_3(\gdl{\forall \vec{x} \varphi(\vec{x})}) \to \forall \vec{x} \PRR_3(\gdl{\varphi(\vec{\dot{x}})})$ for any formula $\varphi(\vec{x})$. 

\begin{proof}
We reason in $T$: 
As in our proof of Claim 5, it suffices to prove the sentence under the assumption ``the bell rings''. 
We assume that the bell rings at Stage $1.m$. 
Suppose $\neg \PRR_3(\gdl{\varphi(\vec{\overline{a}})})$ for some $\vec{a}$. 
Then by Claim 1, $\neg \varphi(\vec{\overline{a}}) \in X_{m-1}$, and hence $\neg \varphi(\vec{\overline{a}}) \in X_{m-1, n}$ for some $n \leq |F_{m - 1}|$ by Lemma \ref{XYC}. 
Since $\varphi(\vec{\overline{a}})$ is an instance of $\forall \vec{x} \varphi(\vec{x})$ and $\neg \forall \vec{x} \varphi(\vec{x}) \in F_{m-1}$ by Claim 3, we have $\neg \forall \vec{x} \varphi(\vec{x}) \in X_{m-1, n+1} \subseteq X_{m-1}$. 
Therefore $\neg \PRR_3(\gdl{\forall \vec{x} \varphi(\vec{x})})$ holds by Claim 1 again. 
\end{proof}

\vspace{0.1in}
{\bf Claim 7.}
If $T \vdash \varphi \to \psi$, then $T \vdash \PRR_3(\gdl{\varphi}) \to \PRR_3(\gdl{\psi})$ for any formulas $\varphi$ and $\psi$.  

\begin{proof}
Suppose $T \vdash \varphi \to \psi$. 
Then $T \vdash \neg \psi \to \neg \varphi$. 
It suffices to show that the sentence $\PRR_3(\gdl{\varphi}) \to \PRR_3(\gdl{\psi})$ is provable in $T+$``the bell rings''. 

We reason in $T+$``the bell rings'': 
Suppose that the bell rings at Stage $1.m$ and that $\neg \PRR_3(\gdl{\psi})$ holds. 
Then by Claim 1, $\neg \psi \in X_{m-1}$. 
Thus $\neg \psi \in X_{m-1, n}$ for some $n \leq |F_{m-1}|$ by Lemma \ref{XYC}. 
Let $k$ be the least proof of $\neg \psi \to \neg \varphi$ in $T$. 
Then $k \leq m-1$ by Claim 3 (because $k$ is standard), and hence $\neg \psi \to \neg \varphi \in P_{T, m-1}$. 
We obtain $\neg \varphi \in X_{m-1, n+1} \subseteq X_{m-1}$. 
Therefore $\neg \PRR_3(\gdl{\varphi})$ holds by Claim 1. 
\end{proof}

\vspace{0.1in}
{\bf Claim 8.}
$\PA \vdash \forall x({\sf True}_{\Delta_0}(x) \to \PRR_3(x))$. 

\begin{proof}
We proceed in $\PA +$ ``the bell rings'': 
Assume that the bell rings at Stage $1.m$. 
Let $\varphi$ be any true $\Delta_0$ sentence. 
If $\neg \varphi \in X_{m-1}$, then the bell rings before Stage $1.m$. 
Thus $\neg \varphi \notin X_{m-1}$. 
By Claim 1, $\neg \PRR_3(\gdl{\varphi})$ does not hold. 
This means $\PRR_3(\gdl{\varphi})$ holds. 
\end{proof}

Our proof of Theorem \ref{DCT3} is completed. 
\end{proof}

\bibliographystyle{plain}
\bibliography{ref}

\begin{thebibliography}{10}

\bibitem{Ara90}
Toshiyasu Arai.
\newblock Derivability conditions on {R}osser's provability predicates.
\newblock {\em Notre Dame Journal of Formal Logic}, 31(4):487--497, 1990.

\bibitem{BM84}
Claudio Bernardi and Franco Montagna.
\newblock Equivalence relations induced by extensional formulae: classification
  by means of a new fixed point property.
\newblock {\em Fundamenta Mathematicae}, 124(3):221--233, 1984.

\bibitem{Buc93}
Wilfried Buchholz.
\newblock Mathematische {L}ogik {II}.
\newblock
  http://www.mathematik.uni-muenchen.de/{$\sim$}buchholz/articles/LogikII.ps,
  1993.

\bibitem{Fef60}
Solomon Feferman.
\newblock Arithmetization of metamathematics in a general setting.
\newblock {\em Fundamenta Mathematicae}, 49:35--92, 1960.

\bibitem{Goed31}
Kurt G{\"o}del.
\newblock {\"U}ber formal unentscheidbare {S\"a}tze der {P}rincipia
  {M}athematica und verwandter {S}ysteme {I}. (in {G}erman).
\newblock {\em Monatshefte f{\"u}r Mathematik und Physik}, 38(1):173--198,
  1931.
\newblock English translation in Kurt G{\"o}del, {\it Collected Works}, Vol. 1
  (pp. 145--195).

\bibitem{GS79}
David Guaspari and Robert~M. Solovay.
\newblock Rosser sentences.
\newblock {\em Annals of Mathematical Logic}, 16(1):81--99, 1979.

\bibitem{HP93}
Petr H{\'a}jek and Pavel Pudl{\'a}k.
\newblock {\em Metamathematics of First-Order Arithmetic}.
\newblock Perspectives in Mathematical Logic. Springer-Verlag, Berlin, 1993.

\bibitem{HB39}
David Hilbert and Paul Bernays.
\newblock {\em Grundlagen der Mathematik. Vol. II}.
\newblock Springer, Berlin, 1939.

\bibitem{HC}
G.~E. Hughes and M.~J. Cresswell.
\newblock {\em A new introduction to modal logic}.
\newblock Routledge, London, 1996.

\bibitem{Jer73}
Robert~G. Jeroslow.
\newblock Redundancies in the {H}ilbert-{B}ernays derivability conditions for
  {G\"o}del's second incompleteness theorem.
\newblock {\em The Journal of Symbolic Logic}, 38(3):359--367, 1973.

\bibitem{KK17}
Makoto Kikuchi and Taishi Kurahashi.
\newblock Universal {R}osser predicates.
\newblock {\em The Journal of Symbolic Logic}, 82(1):292--302, 2017.

\bibitem{Kre60}
Georg Kreisel.
\newblock Ordinal logics and the characterization of informal concepts of
  proof.
\newblock In J.~A. Todd, editor, {\em Proceedings of International Congress of
  Mathematicians 1958}, pages 289--299, New York, 1960. Cambridge University
  Press.

\bibitem{Kre65}
Georg Kreisel.
\newblock Mathematical logic.
\newblock In Thomas~L. Saaty, editor, {\em Lectures in Modern Mathematics},
  volume~3, pages 95--195. Wiley, 1965.

\bibitem{Kre71}
Georg Kreisel.
\newblock A survey of proof theory {II}.
\newblock In Jens~E. Fenstad, editor, {\em Proceedings of the Second
  Scandinavian Logic Symposium}, volume~63 of {\em Studies in Logic and the
  Foundations of Mathematics}, pages 109--170. North Holland, 1971.

\bibitem{KT74}
Gerog Kreisel and Gaisi Takeuti.
\newblock Formally self-referential propositions for cut free classical
  analysis and related systems.
\newblock {\em Dissertationes Mathematicae (Rozprawy Matematyczne)}, 118, 1974.

\bibitem{Kur2}
Taishi Kurahashi.
\newblock A note on derivability conditions.
\newblock In preparation. arXiv:1902.00895.

\bibitem{Kur}
Taishi Kurahashi.
\newblock Rosser provability and normal modal logics.
\newblock {\em Studia Logica}.
\newblock doi: 10.1007/s11225-019-09865-2.

\bibitem{Kur16}
Taishi Kurahashi.
\newblock Henkin sentences and local reflection principles for {R}osser
  provability.
\newblock {\em Annals of Pure and Applied Logic}, 167(2):73--94, 2016.

\bibitem{Lin03}
Per Lindstr{\"o}m.
\newblock {\em Aspects of Incompleteness}.
\newblock Number 10 in Lecture Notes in Logic. A\/K\/Peters, 2nd edition, 2003.

\bibitem{Lob55}
Martin~Hugo L{\"o}b.
\newblock Solution of a problem of {L}eon {H}enkin.
\newblock {\em The Journal of Symbolic Logic}, 20(2):115--118, 1955.

\bibitem{Mon79}
Franco Montagna.
\newblock On the formulas of {P}eano arithmetic which are provably closed under
  modus ponens.
\newblock {\em Bollettino dell'Unione Matematica Italiana}, 16(B5):196--211,
  1979.

\bibitem{Mos65}
Andrzej Mostowski.
\newblock Thirty years of foundational studies: lectures on the development of
  mathematical logic and the study of the foundations of mathematics in
  1930-1964.
\newblock In {\em Acta Philosophica Fennica}, volume~17, pages 1--180. 1965.

\bibitem{Rau10}
Wolfgang Rautenberg.
\newblock {\em A concise introduction to mathematical logic. {T}hird edition}.
\newblock Universitext. Springer, New York, 2010.

\bibitem{Ros36}
John~Barkley Rosser.
\newblock Extensions of some theorems of {G\"o}del and {C}hurch.
\newblock {\em The Journal of Symbolic Logic}, 1(3):87--91, 1936.

\bibitem{Sha91}
Vladimir~Yurievich Shavrukov.
\newblock On {R}osser's provability predicate.
\newblock {\em Zeitschrift f{\"u}r Mathematische Logik und Grundlagen der
  Mathematik}, 37(4):317--330, 1991.

\end{thebibliography}

\end{document}